\documentclass[10pt,a4paper,twoside, reqno, tbtags]{amsart}

\usepackage{amsfonts, amssymb, amscd, amsmath, enumerate, verbatim}
\usepackage{amsaddr}
\usepackage{mathrsfs}
\newtheorem{theorem}{Theorem}[section]
\usepackage{amsthm}
\usepackage{amsfonts}
\usepackage{mathrsfs,graphicx}
\usepackage{xcolor}
\usepackage{breqn}
\usepackage{lipsum}
\usepackage{array}

\usepackage[backend=bibtex]{biblatex}
\newtheorem{lemma}[theorem]{\bf Lemma}%[section]

\newtheorem{question}[theorem]{\bf Question}

\newtheorem{remark}[theorem]{\bf Remark}%[section]
\newtheorem{definition}[theorem]{\bf Definition}%[section]
\newtheorem{example}[theorem]{\bf Example} 
 
\newtheorem{problem}[theorem]{\bf Problem} 
\newcommand\numberthis{\addtocounter{equation}{1}\tag{\theequation}}

\def \ds{\displaystyle}
\keywords{Joint Complete Monotonocity, Moment Problem, Weighted Shift}
\begin{document}
\title{Joint Complete Monotonicity of reciprocal of a polynomial in two variables }

\author{Mandar N.  Khasnis\textsuperscript{1}  \and Vinayak M. Sholapurkar\textsuperscript{2,*}}
%\author{Author1\textsuperscript{1}, Author2\textsuperscript{1}\and Author3\textsuperscript{2,*}}
%\affilOne{\textsuperscript{1} Department of P, University X\\}

\address[1]{\textsuperscript{1}Department of Mathematics, Smt. CHM College, Ulhasnagar, Maharashtra, India, 421 003\\Email: mkhasnis.official@gmail.com\\}

\address[2]{\textsuperscript{2}Bhaskaracharya Pratishthana, Pune, Maharashtra, India, 411004\\
	\textsuperscript{*}Corresponding author.\\
	Email: vmshola@gmail.com}

%\affilTwo{\textsuperscript{2} Department of Q, University Z}

\begin{abstract}
	In this article, we study some special cases of the problem of classifying polynomials $p:\mathbb{R}^2_+\to (0,\infty)$ for which the net $\{\frac{1}{p(m,n)}\}_{m,n\in \mathbb{Z}_+}$ is a  completely monotone net, where $p(x,y)=b(x)+a(x)y$, $a(x)$ and $b(x)$ are polynomials with $\deg a < \deg b$. We also give examples of  $a(x)$ and $b(x)$ such that the net $\{\frac{1}{p(m,n)}\}_{m,n\in \mathbb{Z}_+}$ is not completely monotone. Furthermore, we also study 
    %the operator theoretic manifestation of these results in terms of the Cauchy dual of some torally expansive, toral $k$ isometric weighted $2$-shifts and 
    some properties of the associated subnormal weighted $2$-shifts. 
\end{abstract}
\maketitle	
\section{Introduction}
The multivariable Hausdorff moment problem  is about the characterization of those nets $\{\beta(n)\}_{n\in \mathbb{Z}_+^m}$ for which we have  $$ \beta(n)=\int_{[0,1]^m} x^n d\mu(x),\ \forall n\in \mathbb{Z}_+^m$$ for some finite positive regular Borel measure $\mu$ on $[0,1]^m.$ Such a net is often called as the  {\it Hausdorff moment net}. The solution of the said problem states that a net $\{\beta(n)\}_{n\in \mathbb{Z}_+^m}$ is an Hausdorff moment net if and only if it is {\it joint completely monotone} \cite[Proposition 4.6.11]{bcr1984}. In the case of a single variable, it is well known that if $p(x)=ax+b$ is a positive polynomial, then the sequence $\Big\{\displaystyle \frac{1}{p(n)}\Big\}_{n\in\mathbb{Z}_+}$  is completely monotone \cite{athavale1996}. Further, it follows from \cite[Proposition 6]{AS} that if $p(x,y)$ is a positive polynomial in two variables with degree at most one, then  $\Big\{\displaystyle \frac{1}{p(m,n)}\Big\}_{m,n\in \mathbb{Z}_+}$ is a joint completely monotone net. \\ 
In the light of the above discussion, we naturally have the following problem. 
	    \begin{problem}\label{p1} Classify polynomials $p:\mathbb{R}^2_+\to (0,\infty)$ for which the net $$\Big\{\frac{1}{p(m,n)}\Big\}_{m,n\in \mathbb{Z}_+}$$ is joint completely monotone. %where 
    \end{problem}
    The solution of the problem as stated above is certainly out of sight at present. However, some special cases have been dealt with in the literature. See \cite{nailwal2023_2}, \cite{CN2023},  \cite{nailwalnyj2024}. 
    In this article, we provide a solution in a few more special cases. In particular, we extend the work carried out in \cite{nailwal2023_2} and it also has a bearing on the work appearing in \cite{nailwalnyj2024}. In what follows, we take up the Problem \ref{p1} in the context of polynomials of the type $$p(x,y)=b(x)+a(x)y$$  where $a(x)$ and $b(x)$ are polynomials satisfying $\deg a < \deg b$.
    In particular we give a solution to the problem if $\deg b=\deg a+1$ and also provide an example of a  family of polynomials $p(x,y)$ with $\deg b=\deg a+2$ such that the net $\Big\{\frac{1}{p(m,n)}\Big\}_{m,n\in \mathbb{Z}_+}$ is not joint completely monotone. \\ 
    The Problem \ref{p1} can also be rephrased in terms of the Cauchy dual subnormality problem (CDSP, for short) in the theory of operators on Hilbert spaces. The said problem has been extensively studied in the literature. The interested readers may refer to \cite{acjs2019}, \cite{sc2020}, \cite{athavale1996}, \cite{badea2019}, \cite{cgr2022}, \cite{CJJS2021}, \cite{mkvms2025}, \cite{shimorin2001}. The Problem \ref{p1} can be reformulated in terms of CDSP as given below. 
    \begin{problem}\label{p2}
      Characterize torally expansive toral joint $m-$ isometric multi-shifts whose Cauchy dual is jointly subnormal. 
    \end{problem}
    R. Nailwal has given a complete solution to this problem for $2$-shifts in the case $m=3$. See \cite{nailwalnyj2024}. \\
    In section 2, we explain the notation and terminology used in the sequel. In section 3, we summarize the present status of the problem under consideration and describe the main theorem proved in this article. Section 4 deals with the proof of the main theorem and in the section 5, we shall make some observations regarding the jointly subnormal operator pairs arising out of these considerations.    
    
	\section{Notations and Terminology}
				\noindent  Let $\mathbb{N}$, $\mathbb{Z}_+$ and $\mathbb{R}_+$ respectively denote the set of natural numbers, the set of non negative integers and the set of non negative real numbers. Let $n\in \mathbb{N}$ and  $\alpha=(\alpha_1, \dots,\alpha_n), \beta=(\beta_1,\dots,\beta_n)\in \mathbb{Z}_+^n$. The symbol $|\alpha|$ denotes the sum $\alpha_1+ \dots+\alpha_n.$ We set $(\beta)_\alpha=\prod_{i=1}^n(\beta_i)_{\alpha_i}$, where $(\beta_i)_0=1, (\beta_i)_1=\beta_i$ and \[(\beta_i)_{\alpha_i}=\beta_i(\beta_i-1)\dots(\beta_i-\alpha_i+1), \alpha_i\geq 2, i=1,2,\dots,n.\]
	We write $\alpha\leq \beta$ if $\alpha_i\leq \beta_i$ for every $i=1,\dots,n$. For $\alpha\leq \beta$,  ${\beta \choose \alpha}$ denotes the product $\prod_{i=1}^{n}{\beta_i \choose \alpha_i}$.\\
	\noindent For a net $\{x_\alpha\}_{\alpha\in \mathbb{Z}_+^n}$ and $i=1,\dots,n$, let $\Delta_i$ denote the the {\it forward difference operator} defined as \[\Delta_ix_\alpha=x_{\alpha+\varepsilon_i}-x_\alpha,~~ \alpha\in \mathbb{Z}_+^n,\]
	where $\varepsilon_i$ denotes the $n$-tuple with $i$-th entry equal to $1$ and $0$ elsewhere.
	Note that for any $i,j\in \{1,\dots,n\}$, $\Delta_i\Delta_j=\Delta_j\Delta_i$. For $\alpha=(\alpha_1, \dots,\alpha_n)\in\mathbb{Z}_+^n$, let $\Delta_\alpha$ denote the operator $\prod_{i=1}^{n}\Delta_j^{\alpha_j}$.\\
    
	\noindent The definitions relevant to the present discussion have been recorded below: \\
	\begin{enumerate}
		\item [(i)] A net $a=\{a_\alpha\}_{\alpha\in \mathbb{Z}_+^n}$ is said to be {\it joint completely monotone} if 
	\[(-1)^{|\beta|}\Delta^\beta a_\alpha\geq 0,~~\alpha,\beta \in  \mathbb{Z}_+^n. \]		
	When $n=1$ we refer to $a$ as a {\it completely monotone sequence}. Further, we say that $a$ is {\it separate completely monotone} net if for every $j=1,\dots,n$, $k\in \mathbb{Z}_+$,
	\[(-1)^k\Delta_j^k a_\alpha\geq 0, ~~\alpha \in \mathbb{Z}_+^n.\] 	
	\noindent We note an observation that every joint completely monotone net is separate completely monotone.
\item [(ii)] An infinitely differentiable function $f:\mathbb{R}^n_+\to (0,\infty)$ is said to be {\it joint completely monotone} function if \[(-1)^{|\beta|}(\partial ^\beta f)(x)\geq 0,~~~ \beta\in \mathbb{Z}^n_+, x\in \mathbb{R}^n_+\]
\item [(iii)] A polynomial $p$ in two variables $x_1$ and $x_2$ is said to be of {\it bi-degree}\\ $\alpha=(\alpha_1,\alpha_2)\in \mathbb{Z}_+^2$ if for each $j=1,2$, $\alpha_j$ is the largest integer for which $\partial_j^{\alpha_j}p\neq 0$, where $\partial_j^{\alpha_j}$ denote the partial derivative of order $\alpha_j$ of $p$ with respect to $x_j$.\\
	\end{enumerate}
%	\noindent We now discuss a one variable example. 
%	\begin{example}\label{example1}
%		For $t\in (0,1)$ and $c>0$, consider the sequence $x_m=t^{cm}$, $m\in \mathbb{Z}_+$. Note that
%		\[(-1)^k\Delta^kx_m=(-1)^k\Delta^k(t^{cm})=t^{cm}(1-t^c)^k\geq 0, ~~m,k\in \mathbb{Z}_+.\]
%		This shows that sequence $\{x_m\}_{m\in \mathbb{Z}_+}$ is completely monotone.
%	\end{example}	
%	\noindent The following result about joint completely monotone nets plays crucial role in this paper.
%	\begin{result}\label{res1}
%		The product of two joint completely monotone nets is joint completely monotone. \cite[Chapter 4, Theorem 6.5]{bcr1984}.
%	\end{result}

	\section{The Main Theorem}
	\noindent The complexity of  Problem \ref{p1} (or equivalently, the Problem \ref{p2}) depends on the bi-degree of the polynomial $p(x,y).$ Therefore, it is instructive to examine the nature of the solution in some specific cases.  We therefore focus on the following special case of Problem \ref{p1} by specializing to the polynomial in two variables with bi-degree $(k,1)$:	
    	\begin{problem}\label{p3}
		Let $k,l\in \mathbb{N}$. 
		For $a_i,b_j\in (0,\infty)$, $i=0,1,\dots,l$ and \\ $j=0,1,2,\dots,k$, $(l\leq k)$ let $$a(x)=a_0\ds\prod_{j=1}^{l}(x+a_j),~  ~b(x)=b_0\ds\prod_{j=1}^{k}(x+b_j),~~~x\in \mathbb{R}_+$$
		Characterize those polynomials $p:\mathbb{R}^2_+\to (0,\infty)$ of the form $p(x,y)=b(x)+a(x)y$ for which the net $\left\{\frac{1}{p(m,n)}\right\}_{m,n\in \mathbb{Z}_+}$ is joint completely monotone. 
	\end{problem}  
    
This problem has been studied in \cite{nailwal2023_2,nailwalnyj2024}. The Table \ref{t1} depicts the work carried in these papers and thus provides a picture at a glance about various cases studied in the said papers.
\begin{table}[hbt!] \caption{\bf Various cases studied in literature} 
	\centering
	\begin{tabular}{|c| c| c| c|}  \hline 
		{\bf Bidegree}&{\bf deg(a)}&{\bf deg(b)}&{\bf Studied in}\\ [0.5ex] 
		\hline\hline
		$(k,1)$&$k$&$k$&\cite[Theorem 1.4]{nailwal2023_2}\\
		\hline
		$(1,1)$&$1$&$1$&\cite[Theorem 3.1]{nailwal2023_2}\\
		\hline
		$(2,1)$&$1$&$2$&\cite[Theorem 2.1]{nailwalnyj2024}\\ \hline
		
	\end{tabular}\\
	
	\label{t1}
\end{table}

In the present work, the results from \cite{nailwal2023_2} are extended for certain cases. Indeed, \cite{nailwal2023_2} deals with the polynomials $p(x,y)=b(x)+a(x)y$ where $a(x)$ and $b(x)$ are polynomials of {\it same} degree. It is observed there that one of the necessary conditions for joint complete monotonicity of the net $\{\frac{1}{p(m,n)}\}_{m,n\in \mathbb{Z}_+}$ is $\deg a \leq \deg b$. Here we take up the case where $\deg a < \deg b.$   Further, in the cases where a sufficient condition is satisfied by the net $\{\frac{1}{p(m,n)}\}_{m,n\in \mathbb{Z}_+}$, we apply \cite{shields1974} and \cite{jewell1979} and study some properties of the corresponding subnormal weighted shifts. \\

For the purpose of ready reference, we now state two results from \cite{nailwal2023_2} which are important in the present context. 
	\begin{theorem}\cite[Theorem 1.4]{nailwal2023_2}\label{theorem2}
			Let $k\in \mathbb{N}$. For distinct $a_j,b_j\in (0,\infty)$, $j=0,1,\dots,k$, let $a(x)=a_0\ds\prod_{j=1}^{k}(x+a_j),~  ~b(x)=b_0\ds\prod_{j=1}^{k}(x+b_j),~~~x\in \mathbb{R}_+$.
			Let $p:\mathbb{R}^2_+\to (0,\infty)$ defined as $p(x,y)=b(x)+a(x)y$. \\
			The following statements are valid:
			\begin{enumerate}
				\item[(i)] if $b_1\leq a_1\leq b_2 \leq a_2 \leq \dots \leq b_k\leq a_{k}$ then $\Big\{\ds\frac{1}{p(m,n)}\Big\}_{m,n\in \mathbb{Z}_+}$ is joint completely monotone net.
				\item[(ii)] if $\Big\{\ds\frac{1}{p(m,n)}\Big\}_{m,n\in \mathbb{Z}_+}$ is joint completely monotone net then 
				\[\displaystyle\sum_{j=1}^{k}\displaystyle\frac{1}{a_j}\leq \displaystyle\sum_{j=1}^{k}\displaystyle\frac{1}{b_j},\hspace*{1cm}\displaystyle\prod_{j=1}^{k}b_j\leq \displaystyle\prod_{j=1}^{k}a_j,\hspace*{1cm}\displaystyle\sum_{j=1}^{k}b_j\leq \displaystyle\sum_{j=1}^{k}a_j\]
			\end{enumerate} 
	\end{theorem}  
	\begin{theorem}\cite[Theorem 2.1]{nailwalnyj2024}\label{theorem3}
		Let $p:\mathbb{R}^2_+\to (0,\infty)$ be a polynomial given by $p(x,y)=b(x)+a(x)y$ where $a(x)=a_0(x+a_1)$ and $b(x)=b_0(x+b_1)(x+b_2)$, $a_0,a_1,b_0,b_1,b_2\in \mathbb{R}$, with $b_1\leq b_2$ and $a_0a_1\neq 0$. then the net $\left\{\frac{1}{p(m,n)}\right\}_{m,n\in \mathbb{Z}_+}$ is joint completely monotone if and only if $b_1\leq a_1\leq b_2$.
	\end{theorem}
	
	In the present article, we obtain sufficient condition and a necessary condition for the net in Problem \ref{p3} when $l<k$. The main theorem proved in the article is stated below.
	\begin{theorem}{\normalfont(Main theorem).}\label{theorem1}
		Let $l,k\in \mathbb{N}$ be such that $l<k$. For distinct $a_i,b_j\in (0,\infty)$, $i=0,1,\dots,l$ and $j=0,1,2,\dots,k$, let 
		\begin{equation}\label{eq4}
			a(x)=a_0\ds\prod_{j=1}^{l}(x+a_j),~  ~b(x)=b_0\ds\prod_{j=1}^{k}(x+b_j),~~~x\in \mathbb{R}_+
		\end{equation}
		Let $p:\mathbb{R}^2_+\to (0,\infty)$ be a polynomial in two variables with bi-degree $(k,1)$ defined as $p(x,y)=b(x)+a(x)y$.
		\begin{enumerate}
			\item [(i)] For $l=k-1$;\\ If  $b_1\leq a_1\leq b_2 \leq a_2 \leq \dots \leq a_{l}\leq b_k$ then $\Big\{\ds\frac{1}{p(m,n)}\Big\}_{m,n\in \mathbb{Z}_+}$ is joint completely monotone net.
			\item [(ii)]If $\Big\{\ds\frac{1}{p(m,n)}\Big\}_{m,n\in \mathbb{Z}_+}$ is joint completely monotone net then $$\ds\sum_{j=1}^{l}\ds\frac{1}{a_j}\leq \ds\sum_{j=1}^{k}\ds\frac{1}{b_j}$$\\
		\end{enumerate}
	\end{theorem}
	The proof of the above theorem will be given in the next section. 
	
	The Theorem \ref{theorem1} prompts us to consider the case when $l<k-1.$ In the next section, for  $l=1$ and $k=3$, we provide a family of nets of the form $\displaystyle\Big\{\frac{1}{p(m,n)}\Big\}_{m,n\in \mathbb{Z}_+}$ which are not joint completely monotone, where, \[p(x,y)=b(x)+a(x)y,\\
	b(x)=b_0(x+b_1)(x+b_2)(x+b_3) \text{ and } a(x)=a_0(x+a_1),\] $a_i,bj\in (0,\infty)$, $i=0,1$, $j=0,1,2,3$. \\
	
	\noindent Finally, we fix $n\in \mathbb{N}$ and take up the study of the  weighted sequence space (refer \cite{shields1974}) associated with the sequence $\{\beta_m\}_{m\in \mathbb{Z}_+}$ defined as $\beta_m=\ds\frac{1}{p(m,n)}$ where $p:\mathbb{R}^2_+\to (0, \infty)$ is a polynomial given by $p(x,y)=b(x)+a(x)y$, with\\
	$a(x)=a_0\ds\prod_{j=1}^{k-1}(x+a_j),~  ~b(x)=b_0\ds\prod_{j=1}^{k}(x+b_j)$ and \\$0<b_1\leq a_1\leq b_2 \leq a_2 \leq \dots \leq a_{k-1}\leq b_k$.	\\

Here, we prove the following properties for the corresponding unilateral weighted shift $T$ defined on the reproducing kernel Hilbert space which we denote by  $ \mathcal{H}^2(\beta_m)$.  
	\begin{enumerate}
		\item The operator $T$ is a subnormal contraction.
		\item The operator $T$ is essentially normal.
		\item The spectrum of $T$ is $\overline{\mathbb{D}}$.
	\end{enumerate}  
	
% {\bf Plan of the paper: }A sizable portion of section \ref{sec2} is occupied by the proof of theorem \ref{theorem1}. The proof relies on variety of lemmas (see Lemmas \ref{lemma1}-\ref{lemma3}). Lemma \ref{lemma1} is about partial fraction decomposition of rational functions in one variable with degree of denominator is one less than that of numerator. Lemma \ref{lemma2} is an identity useful in the proof of the main theorem whereas Lemma \ref{l5} can be used to reduce the two variable Hausdorff moment problem to a one variable case. Lastly, Lemma \ref{lemma3} is stated from \cite[Lemma 2.9]{nailwal2023_2} which is useful in proving the necessary part of theorem \ref{theorem1}.
%	As a last remark to this section, we provide a family of nets which is not joint completely monotone in the case  when $l<k-1$ (in particular, $l=1$ and $k=3$). 
	
%	\noindent Section \ref{sec3} deals with application of an article by Shields \cite{shields1974} to construct a weighted sequence space corresponding to sequence obtained by fixing one variable in the given joint completely monotone net. Such sequence is a moment sequence and hence the corresponding weighted shift, say $T$, is a subnormal contraction. We prove that $T$ is essentially normal and spectrum of $T$ is the closed unit disk.

%	\noindent {\color{red}We conclude the paper with.....}\\

%	\noindent The following section deals with the cases with bi-degree $(k,1)$ and $l<k$.
	
	\section{Two cases with bi-degree $(k,1)$ }\label{sec2}
	
	In this section, we present the proof of Theorem \ref{theorem1}.  We begin with a technical lemma. 
	\begin{lemma}\label{lemma1}
		Let $k\in \mathbb{N}$, $k>1$. For $a_i,b_j\in (0,\infty)$, $i=0,1,\dots, k-1$ and $j=0,1,2,\dots,k$, let $$a(x)=a_0\ds\prod_{i=1}^{k-1}(x+a_i),~  ~b(x)=b_0\ds\prod_{i=1}^{k}(x+b_i),~~~x\in \mathbb{R}_+$$
		Then
		\[\ds\frac{b(x)}{a(x)}=c_0\left(x+c+\ds\sum_{i=1}^{k-1}\ds\frac{A_i}{x+a_i}\right)\]
		where, $c_0=\ds\frac{b_0}{a_0}$, $c=\ds\sum_{i=1}^{k}b_i - \ds\sum_{i=1}^{k-1}a_i$ and\\ $A_i=\ds\frac{\ds\prod_{j=1}^{k}(b_j-a_i)}{\ds\prod_{\substack{j=1\\j\neq i}}^{k-1}(a_j-a_i)}$, for each $i=1,2,\dots,k-1$.
		\end{lemma}
	\begin{proof}
	To find $c$ and $A_i$ such that 
	\begin{equation}\label{eq1}
		b(x)=c_0\left((x+c)a(x)+\ds\sum_{i=1}^{k-1}\frac{A_i a(x)}{x+a_i}\right).
	\end{equation}
	The coefficient of $x^{k-1}$ in the left side of equation (\ref{eq1}) is $\ds\sum_{i=1}^{k}b_i$. Further, the terms $\ds\sum_{i=1}^{k-1}\frac{A_i a(x)}{x+a_i}$  from right side of (\ref{eq1}) have degree less than $k-1$ and hence, the coefficient of $x^{k-1}$ in the right side of (\ref{eq1}) is $c+\ds\sum_{i=1}^{k-1}a_i$. Hence, we must have, $c=\ds\sum_{i=1}^{k}b_i - \ds\sum_{i=1}^{k-1}a_i$.\\
	Now, for each $i=1,2,\dots, k-1$, substituting $x=-a_i$ in equation (\ref{eq1}), we get,\\
	\[b(-a_i)=A_i \ds\prod_{\substack{j=1\\j\neq i}}^{k-1}(a_j-a_i)\]
	Hence, for each $i=1,2,\dots, k-1$ we get, 
	\[A_i=\ds\frac{\ds\prod_{j=1}^{k}(b_j-a_i)}{\ds\prod_{\substack{j=1\\j\neq i}}^{k-1}(a_j-a_i)}\]
	This completes the proof.
	\end{proof}
	\noindent The following identity borrowed from \cite{nailwal2023_2} has been used in the proof of Theorem \ref{theorem1}.
	\begin{lemma}\cite[Eqn.(2.1)]{nailwal2023_2}\label{lemma2}
		For any real number $x>0$,
		\[\ds\frac{(-1)^{k-1}}{(k-1)!}\ds\int_{0}^{1}(\log s)^{k-1}s^{x-1+n}ds=\ds\frac{1}{(n+x)^k},~~k\geq 1,~ n\geq 0.\]
	\end{lemma}
	\noindent The following lemma allows us to transform the  two variable Hausdorff moment problem to the one variable case. The statement of the lemma is modified version of \cite[Lemma 2.3]{nailwal2023_2}. Here we have $\deg (a(x))=k-1<k=\deg(b(x))$ whereas the lemma in the stated paper have $\deg (a(x))=k=\deg(b(x))$.
	\begin{lemma}\label{l5}
		Let $a,b$ be two polynomials defined in the equation (\ref{eq4}) such that $l=k-1$ and let $\{c(m)\}_{m\in \mathbb{Z}_+}$ be a sequence of positive real numbers such that the sequence $\{\frac{c(m)}{a(m)}\}_{m\in \mathbb{Z}_+}$ is completely monotone sequence. Then the net $\{\frac{c(m)}{b(m)+a(m)n}\}_{m\in \mathbb{Z}_+}$ is joint completely monotone if $\{t^{\frac{b(m)}{a(m)}}\}_{m\in \mathbb{Z}_+}$ is a Hausdorff moment sequence for every $t\in (0,1)$.
	\end{lemma}
	\begin{proof}
		Assume that $\{t^{\frac{b(m)}{a(m)}}\}_{m\in \mathbb{Z}_+}$ is a Hausdorff moment sequence for every $t\in (0,1)$. Let $A=a/c$ and $B=b/c$. Note that
		\begin{align}
			\displaystyle\frac{c(m)}{b(m)+a(m)n}&=\displaystyle\frac{1}{B(m)+A(m)n}\nonumber\\
			&=\int_{[0,1]}t^n \displaystyle\frac{t^{\frac{B(m)}{A(m)}-1}}{A(m)} dt, ~~~m,n\in \mathbb{Z}_+.
		\end{align}
		As $\frac{B}{A}=\frac{b}{a}$, it is enough to check that for every $t\in (0,1)$, $\left\{ \displaystyle\frac{t^{\frac{B(m)}{A(m)}-1}}{A(m)}\right\}_{m\in \mathbb{Z}_+}$ is a Hausdorff moment sequence. \\
		Given that the sequence $\{\frac{c(m)}{a(m)}\}_{m\in \mathbb{Z}_+}$ is completely monotone sequence, we get the sequence $\{\frac{1}{A(m)}\}_{m\in \mathbb{Z}_+}$ is completely monotone. Further, by assumption we have $\{t^{\frac{b(m)}{a(m)}}\}_{m\in \mathbb{Z}_+}$ is a Hausdorff moment sequence for every $t\in (0,1)$. Since the product of two completely monotone sequences is completely monotone (see \cite[Lemma 8.2.1(v)]{bcr1984}), the claim is settled.
	\end{proof}
	Finally, we state a lemma borrowed from  \cite[Lemma 2.9]{nailwal2023_2}. 
	\begin{lemma}\label{lemma3}
		For polynomials $a,b:\mathbb{R}_+\to (0,\infty)$ let $p(x,y)=b(x)+a(x)y$, $x,y\in \mathbb{R}_+$. If $\frac{1}{p}$ is a joint completely monotone function then \[a'(x)b(x)\leq a(x)b'(x),~~ x\in \mathbb{R}_+~;~~ \text{ and }~~ \deg a \leq \deg b.\]
	\end{lemma}
	\noindent A routine calculation of partial derivatives of $f$ using induction and then the condition for joint complete monotonicity provides the proof of the above lemma.\\
 As stated earlier, we now present the proof of the main Theorem \ref{theorem1}. The argument in the proof is similar to that in proof II of \cite[Theorem 2.1]{nailwal2023_2}.

\begin{proof}[\unskip\nopunct]
	{\bf Proof of Theorem \ref{theorem1} :}\\
	(i) We use Lemma \ref{lemma1} to write \\
	\begin{equation}\label{eq5}
		\ds\frac{b(x)}{a(x)}=c_0\left(x+c+\ds\sum_{i=1}^{k-1}\ds\frac{A_i}{x+a_i}\right)
	\end{equation}
	where, $c_0=\ds\frac{b_0}{a_0}$, $c=\ds\sum_{i=1}^{k}b_i - \ds\sum_{i=1}^{k-1}a_i$ and\\ $A_i=\frac{\ds\prod_{j=1}^{k}(b_j-a_i)}{\ds\prod_{\substack{j=1\\j\neq i}}^{k-1}(a_j-a_i)}$, for each $i=1,2,\dots,k-1$.\\
	From the given relations $b_1\leq a_1\leq b_2 \leq a_2 \leq \dots \leq a_{k-1}\leq b_k$, it follows that $c>0$ and $A_i<0$, for each $i=1,2, \dots, k-1$.
   \\ Now using equation (\ref{eq5}), we get, 
	\[t^{\frac{b(m)}{a(m)}}=t^{c_0m}t^{c_0c}t^{c_0\sum_{i=1}^{k-1}\frac{A_i}{m+a_i}}=t^{c_0m}t^{c_0c}\prod_{i=1}^{k-1}t^{\frac{c_0A_i}{m+a_i}}.\]
	Note that for any $i=1,2,\dots,k-1$, $m\in \mathbb{Z}_+$ and $t>0$, we have 
		\[t^{\frac{A_i}{m+a_i}}=\ds\sum_{j=0}^{\infty}\ds\frac{(A_i \log t)^j}{j!(m+a_i)^j}.\]
	Applying Lemma \ref{lemma2} to each term in the right side we get, \[t^{\frac{A_i}{m+a_i}}=\int_{[0,1]}s^m\delta_1(ds)+\ds\sum_{j=1}^{\infty}\ds\frac{(A_i \log t)^j}{(j-1)!j!} \int_{[0,1]}(-\log s)^{j-1}s^{a_i-1+m}ds,\]
    where $\delta_1$ is a unit point mass measure at $1$.
	By Dominated Convergence theorem, we obtain the following:
	\[t^{\frac{A_i}{m+a_i}}=\int_{[0,1]}s^m\mu_{i,t}(ds)\]
	where, the measure $\mu_{i,t}$ is given by,
	\[\mu_{i,t}(ds)=\delta_1(ds)+w_i(s,t)ds,\]
	with the weight function $w_i$ is given by
	\[w_i(s,t)=s^{a_i-1}\ds\sum_{j=1}^{\infty}\ds\frac{(A_i \log t)^j(-\log s)^{j-1}}{(j-1)!j!},~~s,t\in (0,1).\]
	As $A_i<0$, for each $i=1,2, \dots, k-1$, we get that for each $t\in (0,1)$, $w_i$ is nonnegative function of $s$ on $(0,1)$. 
	Hence for $t\in (0,1)$, $\{t^{\frac{A_i}{m+a_i}}\}$ is a Hausdorff moment sequence, for every $i=1,2,\dots,k-1$. \\
	Also, it can be seen from the definition that, for each $t\in (0,1)$, the sequence $\{t^{c_0m}\}_{m\in \mathbb{Z}_+}$ is completely monotone. Further, the product of two completely monotone sequences is completely monotone (see \cite[Lemma 8.2.1(v)]{bcr1984}). Thus, we have $\{t^{\frac{b(m)}{a(m)}}\}_{m\in \mathbb{Z}_+}$ is a Hausdorff moment sequence for every $t\in (0,1)$. Finally, by applying Lemma \ref{l5}, we conclude that $\{\frac{1}{b(m)+a(m)n}\}_{m,n\in \mathbb{Z}_+}$ is joint completely monotone net.\\
	
	\noindent (ii) It is given that $\Big\{\ds\frac{1}{p(m,n)}\Big\}_{m,n\in \mathbb{Z}_+}$ is joint completely monotone net. Thus by Lemma \ref{lemma3}, we have 
    $$a'(x)b(x)\leq a(x)b'(x)~,~\forall x\in \mathbb{R}_+.$$ 
    This gives\\
	$a_0b_0\ds\prod_{j=1}^{k}(x+b_j) \left(\sum_{i=1}^{l}\prod_{\substack{j=1\\j\neq i}}^{l}(x+a_j)\right)\leq b_0a_0 \prod_{i=1}^{l}(x+a_i)\left(\sum_{i=1}^{k}\prod_{\substack{j=1\\j\neq i}}^{k}(x+b_j)\right)$\\
	$\ds\frac{\ds\sum_{i=1}^{l}\prod_{\substack{j=1\\j\neq i}}^{l}(x+a_j)}{\ds\prod_{i=1}^{l}(x+a_i)}\leq \ds\frac{\ds\sum_{i=1}^{k}\prod_{\substack{j=1\\j\neq i}}^{k}(x+b_j)}{\ds\prod_{j=1}^{k}(x+b_j)}$  \\
	We put $x=0$ to get,  $\ds\sum_{j=1}^{l}\ds\frac{1}{a_j}\leq \ds\sum_{j=1}^{k}\ds\frac{1}{b_j}$.
\end{proof}
\begin{remark}
    The conditions 
    \[\displaystyle\prod_{j=1}^{k}b_j\leq \displaystyle\prod_{j=1}^{k}a_j,\hspace{1cm}\displaystyle\sum_{j=1}^{k}b_j\leq \displaystyle\sum_{j=1}^{k}a_j\]
    appearing in Theorem \ref{theorem2} \cite[Theorem 1.4]{nailwal2023_2} are necessary in the case $\deg(b)=\deg(a)$. However, as illustrated in the example below, these conditions are not necessary in case of  $\deg(b)=\deg(a)+1$. This fact highlights the difference between the cases $l=k$ and $l<k$.\\
    Let $b_0=1, b_1=1, b_2=3$ and $a_0=1, a_1=2$. We have, $b_1<a_1<b_2$. Thus, by Theorem \ref{theorem3}, the net $\Big\{\displaystyle\frac{1}{b(m)+a(m)n}\Big\}_{m,n\in\mathbb{Z}_+}$ is joint completely monotone whereas $b_1, b_2$ and $a_1$ does not satisfy either of the above two conditions.
\end{remark}
We now turn our attention to the case $l<k-1$. We choose $l=1$ and $k=3$ and try to imitate the proof of Theorem \ref{theorem1}.\\
In this case, we have  \[b(x)=b_0(x+b_1)(x+b_2)(x+b_3) \text{ and } a(x)=a_0(x+a_1)\] where $a_i,b_j\in (0,\infty)$, $i=0,1$, $j=0,1,2,3$. Let \begin{equation}\label{eq6}
	p(x,y)=b(x)+a(x)y.
\end{equation}
It turns out that \\
\[\ds\frac{b(x)}{a(x)}=\ds\frac{b_0}{a_0}\left((x+b_1)(x+b_2+b_3-a_1)+(b_3-a_1)(b_2-a_1)+\ds\frac{\ds\prod_{j=1}^{3}(b_j-a_1)}{x+a_1}\right).\]

At this stage, in the proof of Theorem \ref{theorem1} we used the fact that $\{t^{km}\}_{m\in \mathbb{Z}_+}$ is a Hausdorff moment sequence. However, a routine calculation suggests that the sequence $\{t^{km^2}\}_{m\in \mathbb{Z}_+}$ is not a Hausdorff moment sequence. Thus the argument in the proof of Theorem \ref{theorem1} breaks down in the case where $l<k-1$. The indication prompts us to pose the following question:
\begin{question}
    Given the net $\displaystyle\Big\{\frac{1}{p(m,n)}\Big\}_{m,n\in \mathbb{Z}_+}$ (where $p$ defined in equation (\ref{eq6})) satisfying $0<b_1<a_1<b_2<b_3$, is the net $\displaystyle\Big\{\frac{1}{p(m,n)}\Big\}_{m,n\in \mathbb{Z}_+}$ joint completely monotone?
\end{question}

As a step forward in dealing with this problem, we provide a family of nets of the form $\displaystyle\Big\{\frac{1}{p(m,n)}\Big\}_{m,n\in \mathbb{Z}_+}$ (where $p$ defined in equation (\ref{eq6})) which are not joint completely monotone but the condition $0<b_1<a_1<b_2<b_3$ is not satisfied.
\begin{example}\label{ex1}
	Let $p:\mathbb{R}^2_+\to (0,\infty)$ be defined as in equation (\ref{eq6}). Let $t_1=b_1b_2b_3$ and $t_2=(1+b_1)(1+b_2)(1+b_3)$. If \[\ds\frac{(t_2+3)^2}{\left(1/\sqrt{2}\right)^2}-\ds\frac{(t_1+3/2)^2}{\left(1/2\right)^2}<1\]
	then the net $\{1/p(m,n)\}_{m,n\in \mathbb{Z}_+}$ is not joint completely monotone.
\end{example}
\begin{proof}
	\noindent Let $\beta(m,n)=\ds\frac{1}{p(m,n)}$ for $m,n\in \mathbb{Z}_+$. Then,
we have, \begin{align*}
	\Delta_1\Delta_2\beta(m,n)=&\beta(m+1,n+1)-\beta(m,n+1)-\beta(m+1,n)+\beta(m,n)\\
	=&\ds\frac{1}{(m+1+b_1)(m+1+b_2)(m+1+b_3)+(m+1+a_1)(n+1)}\\&-\ds\frac{1}{(m+b_1)(m+b_2)(m+b_3)+(m+a_1)(n+1)}\\
	&-\ds\frac{1}{(m+1+b_1)(m+1+b_2)(m+1+b_3)+(m+1+a_1)n}\\&-\ds\frac{1}{(m+b_1)(m+b_2)(m+b_3)+(m+a_1)n}.
\end{align*}
Take $t_2=(m+1+b_1)(m+1+b_2)(m+1+b_3), t_1=(m+b_1)(m+b_2)(m+b_3)$, then we get, \\
\begin{align*}
	\Delta_1\Delta_2\beta(m,n)=&\ds\frac{1}{t_2+(m+1+a_1)(n+1)}-\ds\frac{1}{t_2+(m+1+a_1)n}\\&-\ds\frac{1}{t_1+(m+a_1)(n+1)}+\ds\frac{1}{t_1+(m+a_1)n}\\
	=&-\ds\frac{m+1+a_1}{D_2}+\ds\frac{m+a_1}{D_1}
\end{align*}
Where,\\
 $D_2=(t_2+(m+1+a_1)(n+1))(t_2+(m+1+a_1)n)$ and \\ $D_1=(t_1+(m+a_1)(n+1))(t_1+(m+a_1)n)$.\\
Consider the case, $m=0, a_1=1, n=1$:\\
We get, $$\Delta_1\Delta_2\beta(0,1)=\ds\frac{D_2-2D_1}{D_1D_2}.$$
In this case, we note the following: $$t_1=b_1b_2b_3,  t_2=(1+b_1)(1+b_2)(1+b_3)\text{ and } D_2=(t_2+4)(t_2+2) ,D_1=(t_1+1)(t_1+2).$$
Hence, \begin{align*}
	\Delta_1\Delta_2\beta(0,1)<0 &\text{ if } D_2-2D_1<0\\
	&\text{ if } t_2^2-2t_1^2+6t_2-6t_1+4<0\\
	&\text{ if } \ds\frac{(t_2+3)^2}{\left(1/\sqrt{2}\right)^2}-\ds\frac{(t_1+3/2)^2}{\left(1/2\right)^2}<1. \numberthis \label{eq8}
\end{align*}
\end{proof}
\noindent Note that we can choose the values of $b_1, b_2, b_3$ and $a_1$ such that the condition (\ref{eq8}) is satisfied. We obtain two of such families below.
\begin{enumerate}
	\item For any $b\in \mathbb{R}^+$, we set  $b_1=b,b_2=2b,b_3=3b$ and $a_1=1$, the condition (\ref{eq8}) becomes
\[-36b^6+132b^5+193b^4+144b^3+124b^2+48b+11<0.\]
This is true only if $b<-0.849$ or $b>4.94$. \\
As $b>0$, we can conclude that for each value of $b>4.94$, the corresponding net is not joint completely monotone. 
\item For any value of $b\in \mathbb{R}^+$, we set $b_1=b_2=b_3=b$ and $a_1=1$. The condition (\ref{eq8}) becomes
\[-b^6+6b^5+15b^4+20b^3+33b^2+24b+11<0.\]
This is true only if $b<-0.849$ or $b>8.19$. \\
As $b>0$, we can conclude that for each value of $b>8.19$, the corresponding net is not joint completely monotone.\\
\end{enumerate}

\section{Hilbert space associated with a Hausdorff moment sequence(net) and multiplication by $z$ operator }\label{sec3}
It is known that corresponding to every sequence (or net) of positive real numbers, one can associate a reproducing kernel Hilbert space (see  \cite{shields1974}, \cite{jewell1979}). Further, it is a well known fact that if the sequence (or net) is a Hausdorff moment sequence (or net), then the operator multiplication by $z$ ($M_z$) on the corresponding Hilbert space is subnormal. In this section, we study Hilbert spaces that arise in the context of the sequences of the form $\{1/p(m,n)\}_{m\in \mathbb{Z}_+}$, for fixed $n\in \mathbb{Z}_+$ and obtain some operator theoretic properties of the corresponding operator $M_z$. 

 %In the previous section we have seen the sufficient condition for a net to be Hausdorff moment net. We now move our attention to a Hilbert space, known as weighted sequence space, which can be associated to every such net. A classical reference about weighted shift operators and sequence spaces is \cite{shields1974} and its sequel for multivariable version is \cite{jewell1979}. 
 We now record  some definitions and results borrowed from \cite{shields1974} which are required in sequel.

\begin{definition} (\cite[Section 3]{shields1974})
For a sequence of positive real numbers $\{\gamma(n)\}$ with $\gamma(0)=1$, $H^2(\gamma)$ is the space of all sequences $\{f=\hat{f(n)}\}$ such that \\$\|f\|_\gamma^2=\sum |\hat{f(n)}|^2|\gamma(n)|^2<\infty$.
\end{definition}
\begin{definition} (\cite[Section 1]{shields1974})
	Let $\mathcal{H}$ be a separable Hilbert space with an orthonormal basis $\{e_n\}_{n\in \mathbb{Z}_+}$. Let $\{w_n\}_{n\in \mathbb{Z}_+}$ be a sequence of positive real numbers. A unilateral weighted shift $T$ is an operator from $\mathcal{H}$ to itself defined as $Te_n=w_ne_{n+1}$, for all $n\in \mathbb{Z}_+$.
\end{definition}
 We  find it convenient to quote the following two theorems which are relevant in the present context. 
\begin{theorem}(\cite[Proposition 7]{shields1974})\label{thm1}
	\\The linear transformation $M_z$ on $H^2(\gamma)$ is unitarily equivalent to an injective unilateral weighted shift operator (with weight sequence $\{\alpha_m\}$ given below). Conversely, every injective unilateral weighted shift operator $T:\{\alpha_m\}$ is unitarily equivalent to $M_z$ acting on $H^2(\gamma)$, for a suitable choice of $\gamma$). The relation between $\{\alpha_m\}$ and $\gamma$ is given by the equations:
	$$\alpha_m=\ds\frac{\gamma(m+1)}{\gamma(m)}\hspace*{1cm} \text{and}\hspace*{1cm}
	\gamma(m)=\alpha_0\alpha_1\dots \alpha_{m-1}.$$
\end{theorem}
\begin{theorem}(\cite[Proposition 25]{shields1974})\label{thm2}
	Let $T:\{\alpha_m\}$ be an injective unilateral weighted shift represented as $M_z$ on $H^2(\gamma)$. Then $T$ is subnormal contraction if and only if $\{\beta_m\}$ is a moment sequence, where $\beta_m=\gamma_m^2=\alpha_0^2\alpha_1^2\dots \alpha_{m-1}^2$.
\end{theorem}
 
%\noindent Before we work with the Hausdorff moment sequences from the previous section, we state a result which provides a necessary and sufficient condition for a certain type of net to be Hausdorff moment net.
%\begin{theorem}(\citep[Theorem 1.2]{nailwal2023})\label{thm3}
%	Let $p:\mathbb{R}^2_+\to (0, \infty)$ be a polynomial given by $p(x,y)=b(x)+a(x)y$, where $a(x)=a_0(x+a_1)$ and $b(x)=b_0(x+b_1)(x+b_2)$, $a_0, a_1, b_0, b_1, b_2\in \mathbb{R}$, with $b_1\leq b_2$ and $a_0, a_1\neq 0$. Then the net $\Big\{\ds\frac{1}{p(m,n)}\Big\}_{m,n\in \mathbb{Z}_+}$ is joint completely monotone if and only if $b_1\leq a_1 \leq b_2$.
%\end{theorem}
\vspace*{0.3cm}
\noindent For a fixed $n\in \mathbb{Z}_+,$ we take up the study of sequence \begin{equation}\label{eq7}
	\{\beta_m\}_{m\in \mathbb{Z}_+}
\end{equation} defined as $\beta_m=\ds\frac{1}{p(m,n)}$ where $p:\mathbb{R}^2_+\to (0, \infty)$ is a polynomial given by $p(x,y)=b(x)+a(x)y$  with
$a(x)=a_0\ds\prod_{j=1}^{k-1}(x+a_j),~  ~b(x)=b_0\ds\prod_{j=1}^{k}(x+b_j)$ and $0<b_1\leq a_1\leq b_2 \leq a_2 \leq \dots \leq a_{k-1}\leq b_k$. \\

\noindent Observe that  by Theorem \ref{theorem1}, the net $\ds\Big\{\frac{1}{p(m,n)}\Big\}_{m,n\in \mathbb{Z}_+}$ is joint completely monotone which in turn implies that the sequence $\{\beta_m\}$ is Hausdorff moment sequence. \\

Let $T:\{\alpha_m\}$ be the unilateral weighted shift where $\alpha_m=\sqrt{\ds\frac{\beta_{m+1}}{\beta_m}}$. Theorem \ref{thm2} implies that  $T$ is a subnormal contraction.
\\Further by applying Theorem \ref{thm1}, the above unilateral weighted shift $T:\{\alpha_m\}$ becomes unitarily equivalent to $M_z$ on $H^2(\gamma)$ where $$\gamma_m=\alpha_0\alpha_1\dots\alpha_{m-1}=\sqrt{\ds\frac{\beta_m}{\beta_0}}.$$\\

A couple of observations about the operator $T$ have been recorded below:
\begin{theorem}
	The operator $T$ is essentially normal.
\end{theorem}
\begin{proof}
	We know that (ref. \cite[Proposition 6.3]{conway1991}), for a unilateral weighted shift $T:\{\alpha_m\}$, \[(T^*T-TT^*)e_0=\alpha_0^2e_0 \text{ and }\] \[(T^*T-TT^*)e_m=(\alpha^2_m-\alpha^2_{m-1})e_m, \text{ for all }m\in \mathbb{N}.\]
As $\alpha_m=\sqrt{\ds\frac{\beta_{m+1}}{\beta_m}}$,  
We get, $\alpha^2_m=\ds\frac{\beta_{m+1}}{\beta_m}=\ds\frac{p(m,n)}{p(m+1,n)}\to 1$ as $m\to \infty$.\\
	Thus, $\alpha^2_m-\alpha^2_{m-1}\to 0$ as $m\to \infty$.\\
	Hence, $T$ is a subnormal operator which is essentially normal.
\end{proof}
\begin{theorem}
	The spectrum of $T$ is $\overline{\mathbb{D}}$.
\end{theorem}
\begin{proof}
	Let $r(T)$ denote the spectral radius of an operator $T$. We know that the spectrum of a unilateral shift $T$ is the disc $\{z\in \mathbb{C}: |z|\leq r(T)\}$ (\cite[Theorem 4]{shields1974}). 
	For any $m\in \mathbb{Z}_+$, $\|T^m\|=\displaystyle\sup_{i\in \mathbb{Z}_+} | \alpha_i \alpha_{i+1}\alpha_{i+2}\dots \alpha_{i+m-1}|$(\cite[Proposition 2]{shields1974}).\\
	Thus, we get, 
	\begin{align*}
		\|T^m\|&=\displaystyle\sup_{i\in \mathbb{Z}_+}\sqrt{\displaystyle\frac{\beta_{i+1}}{\beta_i}\frac{\beta_{i+2}}{\beta_{i+1}}\dots\frac{\beta_{i+m}}{\beta_{i+m-1}}}\\
		 &=\displaystyle\sup_{i\in \mathbb{Z}_+}\sqrt{\displaystyle\frac{\beta_{i+m}}{\beta_i}}\\
		 &=\displaystyle\sup_{i\in \mathbb{Z}_+}\sqrt{\frac{p(i,n)}{p(i+m,n)}}\\
		 &=\displaystyle\sup_{i\in \mathbb{Z}_+}\sqrt{\frac{b(i)+a(i)n}{b(i+m)+a(i+m)n}}.			
	\end{align*}
	
	Using the fact that $r(T)=\displaystyle\lim_{m\to \infty}\|T^m\|^{1/m}$ we get, $r(T)=1$ and hence, spectrum of $T$ is $\overline{\mathbb{D}}$.
	
\end{proof}
\noindent Finally, by choosing the positive real numbers $a_0,a_1,b_0,b_1,b_2$ such that $b_1\leq a_1\leq b_2$, we get the sequences defined by the equation (\ref{eq7}), that give us a family of subnormal contractions. We note that the subnormal operators so obtained are not necessarily unitarily equivalent.
\begin{example}{Subnormal operators which are not unitarily equivalent:}
\end{example}
\noindent Consider $a_0,a_1,b_0,b_1,b_2\in \mathbb{R}_+$ such that $b_1\leq a_1\leq b_2$. \\
Let $p(x,y)=b(x)+a(x)y$, $b(x)=b_0(x+b_1)(x+b_2)$ and $a(x)=a_0(x+a_1)$. By Theorem \ref{theorem1}, we know that the net $\Big\{\ds\frac{1}{p(m,n)}\Big\}_{m,n\in \mathbb{Z}_+}$ is joint completely monotone net and thus for fixed $n\in \mathbb{N}$, $\{\beta_m\}$ is a Hausdorff moment sequence where $\beta_m=\ds\frac{1}{p(m,n)}$.\\
Following the discussion after Theorem \ref{thm2}, we have a unilateral weighted shift $T:\{\alpha_m\}$ which is a subnormal contraction and is unitarily equivalent to $M_z$ on $H^2(\gamma)$ where $\gamma_m=\alpha_0\alpha_1\dots\alpha_{m-1}=\sqrt{\ds\frac{\beta_m}{\beta_0}}$.
\\
We know that (Refer \cite[Equation (14)]{shields1974}), for any $f\in H^2(\gamma)$, $$\|f\|^2_\gamma=\sum|\hat{f}(m)|^2\gamma^2_m=\sum|\hat{f}(m)|^2\left(\ds\frac{\beta_m}{\beta_0}\right).$$\\
Thus, for $f(z)=z$, $\|z\|_\gamma=\ds\frac{\beta_1}{\beta_0}=\ds\frac{b_0b_1b_2+a_0a_1n}{b_0(1+b_1)(1+b_2)+a_0(1+a_1)n}.$\\

This justifies that for different choices of $a_1,b_1,b_2$ with $b_1\leq a_1\leq b_2$, $M_z$ on corresponding $H^2(\gamma)$ are all subnormal contractions but they need not be unitarily equivalent to each other.

\section{Conclusion}
The present article deals with only a few cases of a larger problem as stated in Problem \ref{p3}. Authors believe that the solution of the general problem demands the use of techniques and methods beyond the one used in the present article. We propose to continue the study of the sequences generated by taking reciprocals of two variable polynomials and check their joint complete monotonicity. As highlighted in the present paper, such a study also allows one to talk about a variety of subnormal contractive operators.
\section*{Declaration}
The present work is carried out at the research center at the Department of Mathematics, S. P. College, Pune, India(autonomous).

\printbibliography
%\bibliographystyle{plain}
%\bibliography{ref_JCM.bib}{}
\end{document}